\documentclass[12pt]{amsart}
\usepackage{amscd,amsmath,amsthm,amssymb}
\usepackage[left]{lineno}
\usepackage{pstcol,pst-plot,pst-3d}
\usepackage{color}
\usepackage{pstricks}
\usepackage{stmaryrd}
\newpsstyle{fatline}{linewidth=1.5pt}
\newpsstyle{fyp}{fillstyle=solid,fillcolor=verylight}
\definecolor{verylight}{gray}{0.97}
\definecolor{light}{gray}{0.9}
\definecolor{medium}{gray}{0.85}
\definecolor{dark}{gray}{0.6}

 \def\NZQ{\mathbb}               

 \def\ZZ{{\NZQ Z}}
 \def\RR{{\NZQ R}}
 \def\CC{{\NZQ C}}
 
 \def\PP{{\NZQ P}}

 \def\frk{\mathfrak}               

 \def\mm{{\frk m}}

 \def\Ic{{\mathcal I}}

 \def\G{{\mathcal G}}

 \def\ab{{\mathbf a}}
 
 \def\xb{{\mathbf x}}

 \def\opn#1#2{\def#1{\operatorname{#2}}} 
 \opn\chara{char} \opn\length{\ell} \opn\pd{pd} \opn\rk{rk}
 \opn\projdim{proj\,dim} \opn\injdim{inj\,dim} \opn\rank{rank}
 \opn\depth{depth} \opn\grade{grade} \opn\height{height}
 \opn\embdim{emb\,dim} \opn\codim{codim}
 
 \opn\Tr{Tr} \opn\bigrank{big\,rank}
 \opn\superheight{superheight}\opn\lcm{lcm}
 \opn\trdeg{tr\,deg}
 \opn\reg{reg} \opn\lreg{lreg} \opn\ini{in} \opn\lpd{lpd}
 \opn\size{size} \opn\sdepth{sdepth}
 \opn\link{link}\opn\fdepth{fdepth}\opn\lex{lex}
 \opn\tr{tr}
 \opn\type{type}
 \opn\gap{gap}
 \opn\arithdeg{arith-deg}
 \opn\div{div} \opn\Div{Div} \opn\cl{cl} \opn\Cl{Cl}
 \opn\Spec{Spec} \opn\Supp{Supp} \opn\supp{supp} \opn\Sing{Sing}
 \opn\Ass{Ass} \opn\Min{Min}\opn\Mon{Mon}
 \opn\Ann{Ann} \opn\Rad{Rad} \opn\Soc{Soc}
 \opn\Im{Im} \opn\Ker{Ker} \opn\Coker{Coker} \opn\Am{Am}
 \opn\Hom{Hom} \opn\Tor{Tor} \opn\Ext{Ext} \opn\End{End}
 \opn\Aut{Aut} \opn\id{id}
 
 \opn\nat{nat}
 \opn\pff{pf}
 \opn\Pf{Pf} \opn\GL{GL} \opn\SL{SL} \opn\mod{mod} \opn\ord{ord}
 \opn\Gin{Gin} \opn\Hilb{Hilb}\opn\sort{sort}
 \opn\PF{PF}\opn\Ap{Ap}
 \opn\mult{mult}
 \opn\aff{aff}
 \opn\relint{relint} \opn\st{st}
 \opn\lk{lk} \opn\cn{cn} \opn\core{core} \opn\vol{vol}  \opn\inp{inp} \opn\nilpot{nilpot}
 \opn\link{link} \opn\star{star}\opn\lex{lex}\opn\set{set}
 \opn\width{wd}
 \opn\Fr{F}
 \opn\QF{QF}
 \opn\G{G}
 \opn\type{type}\opn\res{res}
 \opn\conv{conv}
 \opn\gr{gr}
 
 \def\pot#1#2{#1[\kern-0.28ex[#2]\kern-0.28ex]}

 \opn\dirlim{\underrightarrow{\lim}}
 \opn\inivlim{\underleftarrow{\lim}}
 \let\union=\cup
 \let\sect=\cap
 
 \let\tensor=\otimes
 \let\iso=\cong
 \let\Union=\bigcup
 \let\Sect=\bigcap
 \let\Dirsum=\bigoplus
 
 \let\to=\rightarrow
 
 \def\Implies{\ifmmode\Longrightarrow \else
         \unskip${}\Longrightarrow{}$\ignorespaces\fi}
 \def\implies{\ifmmode\Rightarrow \else
         \unskip${}\Rightarrow{}$\ignorespaces\fi}
 \def\iff{\ifmmode\Longleftrightarrow \else
         \unskip${}\Longleftrightarrow{}$\ignorespaces\fi}

 \let\:=\colon
 \newtheorem{Theorem}{Theorem}[section]
 \newtheorem{Lemma}[Theorem]{Lemma}
 \newtheorem{Corollary}[Theorem]{Corollary}
 \newtheorem{Proposition}[Theorem]{Proposition}
 \newtheorem{Remark}[Theorem]{Remark}

 \newtheorem{Examples}[Theorem]{Examples}
 \newtheorem{Definition}[Theorem]{Definition}

 \let\epsilon\varepsilon
 \let\kappa=\varkappa
 \def\qed{\ifhmode\textqed\fi
       \ifmmode\ifinner\quad\qedsymbol\else\dispqed\fi\fi}
 \def\textqed{\unskip\nobreak\penalty50
        \hskip2em\hbox{}\nobreak\hfil\qedsymbol
        \parfillskip=0pt \finalhyphendemerits=0}
 \def\dispqed{\rlap{\qquad\qedsymbol}}
 
 \opn\dis{dis}
 \def\pnt{{\raise0.5mm\hbox{\large\bf.}}}
 
 \opn\Lex{Lex}

\begin{document}

\title {On the monomial  reduction number of a monomial ideal in $K[x,y]$}

\author {J\"urgen Herzog, Somayeh Moradi,  Masoomeh Rahimbeigi and Ali Soleyman Jahan}

\address{J\"urgen Herzog, Fachbereich Mathematik, Universit\"at Duisburg-Essen, Campus Essen, 45117
Essen, Germany} \email{juergen.herzog@uni-essen.de}

\address{Somayeh Moradi, Department of Mathematics, School of Science, Ilam University,
P.O.Box 69315-516, Ilam, Iran}
\email{somayeh.moradi1@gmail.com}

\address{Masoomeh Rahimbeigi, Department of Mathematics, University of Kurdistan, Post
Code 66177-15175, Sanandaj, Iran}
\email{rahimbeigi-masoome@yahoo.com}

\address{Ali Soleyman Jahan, Department of Mathematics, University of Kurdistan, Post
Code 66177-15175, Sanandaj, Iran}
\email{solymanjahan@gmail.com}

\dedicatory{ }

\begin{abstract}
The reduction  number of monomial ideals in the polynomial $K[x,y]$ is studied. We focus on ideals $I$ for which $J=(x^a,y^b)$ is a reduction ideal. The computation of the reduction number amounts to solve linear inequalities. In some special cases the reduction number can be explicitly computed.
\end{abstract}

\thanks{}




\maketitle

\setcounter{tocdepth}{1}

\section*{Introduction}
Let $K$ be a field. In this paper we study monomial  ideals  in the polynomial ring $K[x,y]$  for which $J=(x^a,y^b)$ is a reduction ideal for suitable integers $a,b\geq 1$. We denote this class of ideals by $\Ic_{a,b}$. Thus $I\in \Ic_{a,b}$ if and only if $I^{r+1}=JI^r$ for some integer $r\geq 0$. The smallest integer $r$,  with this property,  denoted $r(I)$, is called the {\em reduction number} of $I$ with respect to $J$. The main concern of this paper is to bound $r(I)$ and to compute it explicitly in some special cases.

The monomial ideals belonging to $\Ic_{a,b}$ can be described as follows: for $u\in K[x,y]$, $u=x^cy^d$, let $\nu(u)=a/c+b/d$. Then it is easy to see that $I\in \Ic_{a,b}$ if and only if $\nu(u)\geq 1$ for all $u\in I$. We call an ideal $I\in \Ic_{a,b}$ {\em quasi-equigenerated}, if $\nu(u)=1$ for all $u\in G(I)$. Here $G(I)$ is the unique set of minimal monomial generators of $I$. Note that  $I\in \Ic_{a,b}$ is quasi-equigenerated if and only if all monomial generators of $I$ are of degree $ab$ with the respect to the non-standard grading of $K[x,y]$ given by $\deg(x)=b$ and  $\deg(y)=a$. The set of quasi-equigenerated ideals in $\Ic_{a,b}$ we denote by $\Ic_{a,b}^1$. Let $g=\gcd(a, b)$, then $I\in\Ic_{a,b}^1$, if and only if there exists  $A\subseteq  \{1,\ldots,g\}$ with $1,g\in A$ such that $I=I_A$, where $G(I_A)=\{x^{i\frac{a}{g}}y^{b-i\frac{b}{g}}\:\; i\in A\}$.

Vasconcelos \cite{Va}  gives  an  upper bound for the reduction number  of the graded maximal ideal of a standard graded $K$-algebra $A$  in terms of the  arithmetic degree $A$. This bound applied to the fiber $F(I)$ of $I$ yields the inequality $r(I)<\arithdeg(F(I))$. In general the arithmetic
degree is hard to compute, but in the case that $I$ is quasi-equigenerated, the arithmetic degree of $F(I)$ coincides with the multiplicity $e(F(I))$  of $F(I)$. This fact enables us to  show  that $r(I)<g/\gcd(A)$ for $I\in \Ic_{a,b}^1$ where $g=\gcd(a,b)$ and $I=I_A$, see Theorem~\ref{cold}. For the convenience of the reader  we give here a self-contained  proof of the fact that $e(F(I_A)))=g/\gcd(A)$.   By using a strong result of  Gruson, Lazardsfeld and Peskine  \cite{GLP} one obtains an even stronger upper bound, namely  $r(I)\leq g/\gcd(A)-|A|+2$ (Theorem~\ref{GLP}). In the special case considered here this bound is also the consequence of a beautiful result of Hoa and St\"uckrad \cite[Theorem 1.1]{HSt}

In Theorem~\ref{Somayeh} it is shown that for any integers $1\leq a\leq b$ and any $1\leq j\leq g-1$ where $g=\gcd(a,b)$,  there exist $I_A\in \Ic_{a,b}$ with $\gcd(A)=1$ and $r(I_A)=j$. This shows that reduction number can take any value up to the bound given in Theorem~\ref{cold}. In the next result, Theorem ~\ref{redone},   the monomial ideals $I=I_A\in \Ic_{a,b}^1$ for which $r(I)=1$ are characterized, and in Proposition~\ref{masiproves} a precise formula for the reduction number of $3$-generated $I_A\in \Ic_{a,b}^1$ is given. This result is used in Theorem~\ref{sunshine} to classify  the monomial ideals in $\Ic_{a,b}$ which attain the maximal possible value for the reduction number, namely the number $g-1$ where $g=\gcd(a,b)$. We close Section~\ref{value} by discussing  for a given integer $a$ the ratio $m_a(j)/m_a$, where $m_a$ is the cardinality of $\Ic_{a,b}^1$ and $m_a(j)$ is the cardinality of the monomial ideals in  $\Ic_{a,b}^1$ with  $r(I)=j$. We expect that $\lim_{a\to\infty}m_a(j)/m_a$ exists and is equal to $0$. This can be easily shown for $j=1$. On the other hand, if one asks a similar question for $3$-generated ideals, then by using a famous result from number theory due to  Hardy and Wright \cite{HW}, we show in Proposition~\ref{ohohoh} that such a limit does  not always  exist.

In Section~\ref{regensburg} we consider more closely the reduction number of monomial ideals in $\Ic_{a,b}$ which are not necessarily quasi-equigenerated. It turns out, see Proposition~\ref{goulash},  that $\min\{a,b\}-1$ is a bound for reduction number of such ideals. Let $I_p=(x^a,u_p,y^b)$ where $u_p=x^cy^d$ for $p=(c,d)$ and  $\nu(u_p)\geq 1$.  For fixed numbers $a$ and $b$, $r(I_p)$ depends only on position of $p$. For example, $r(I_p)=a-1$ if $p=(1,b-1)$ and $r(I_p)=1$ if and only if $a/2\leq c$ and $b/2\leq d$, see Proposition \ref{freshliver}. Let $R_{(a,b)}$ be the set of numbers $r(I_p)$ with $I_p\in\Ic_{a,b}$. By what we said before it follows that $R_{(a,b)}\subseteq\{1,\ldots, \min\{a,b\}-1\}$. In general, the difference between these  two sets  can be as big as we want. Indeed, in Proposition~\ref{ourlimits} we show that for any prime number $p\geq 2$   we have $(p-1)-|R_{(p,p)}|\geq(p-1)/2-1$. In contrast to this result we show in Proposition~\ref{specialnight} that for any integer $a\geq 2$ we have $\bigcup_b R_{(a,b)}=\{1,\ldots,a-1\}$.

A comparison between the reduction number of an ideal $I\in\Ic_{a,b}$ and its quasi-equigenerated part $I_0\in \Ic_{a,b}^1$ is made in Section~\ref{part}. It is shown in Proposition~\ref{turtle} that $r(I_0)\leq r(I)$. In the case that $I=I_A+I'$ where $I'=(u:\nu(u)>r)$, this comparison shows that the upper  bound given by Gruson, Lazardsfeld and Peskine \cite{GLP} holds for  $r(I)$, even if $I$ is not quasi-equigenerated. This fact  applies in particular to lex ideals in $\Ic_{a,b}$ when $a<b$.

In the following Section~\ref{unique} we prove the surprising fact, shown in Corollary~\ref{minimal}, that for any monomial ideal in $I\in\Ic_{a,b}$ there exist unique  monomial ideal $I\subseteq L$ with $r(L)=1$ and such that $r(L')>1$ for all $L'$ with $I\subseteq L'\subset L$. Then in Theorem~\ref{wetransfer} we compute all the generators of $L$ when $I$ itself is generated by  $3$ elements. Finally in the last section we study the reduction number of powers of ideals belonging to $\Ic_{a,b}$.  By a result of Hoa \cite[Proposition ]{Hoa}, $r(I^k)\leq \lceil (r(I)-1)/k\rceil +1$ for $I\in \Ic_{a,b}$. This implies in particular that $r(I^{kl})\leq r(I^l)$ for all $k,l\geq 1$.  However, we do not know whether  $r(I^{k+1})\leq r(I^k)$ for all $k\geq 0$. In Theorem~\ref{related} all ideals  $I\in \Ic_{a,b}^1$ with the property that $r(I^k)=1$ for $k\gg 0$ are characterized, and it is shown that if $r(I^{k_0})=1$ for some $k_0$, then $r(I^k)=1$ for $k\geq k_0$.
In Proposition~\ref{masoomehproves} we show that for $I\in \Ic_{a,b}^1$ the number $inf\{k:r(I^k)=1\}$ can take any value between $1$ and $\gcd(a,b)-2$, and in Proposition~\ref{hoa} it is shown that for $k\geq a-2$, $r(I^k)$ is constant $1$ or constant $2$ for all $I\in\Ic_{a,a}$, depending on $xy^{a-1},x^{a-1}y$ belong to $I$ or not belong to $I$.

\section{Preliminaries}
\label{preliminaries}

Let $R$ be a Noetherian ring  and $I\subset R$ be an ideal. An ideal $J\subseteq I$ is called a {\em reduction ideal} of $I$, if $I^{k+1}=JI^k$ for some $k\geq 1$.
The smallest number $ k $ for which  $ I^{k+1}=JI^{k} $ is called the  {\em reduction number} of $I$ with respect to $J$, and is denoted by $r _{J}(I) $.

\begin{Definition}
{\em  Let $K$ be a field and $S=K[x_1,\ldots,x_n]$ the polynomial ring over $K$ in $n$ interminates. Let $I\subset S$ be a  monomial ideal. A {\em monomial reduction} of $I$ is a monomial ideal $J\subset I$ which is a reduction ideal of $I$.  A monomial reduction ideal $J$  of $I$ is called a {\em minimal monomial reduction ideal} of $I$, if any monomial ideal $L$ which is properly contained in $J$, is not a reduction ideal of $I$.}
\end{Definition}


In order to  describe the minimal monomial reductions of a monomial ideal, we introduce some notation.
Let $u=x_1^{a_1}x_2^{a_2}\cdots x_n^{a_n}$ be a monomial in $S$. Then $\ab=(a_1,a_2,\ldots,a_n)$ is called the {\em exponent vector} of $u$, and  we write $u=\xb^\ab$.  By a result of Singla \cite[Proposition 2.1]{Si},  $I$ admits a unique minimal monomial reduction. More precisely, she shows the following:
\begin{Proposition}
 \label{pooja}
Let $I\subset K[x_1,\ldots,x_n]$ be a monomial ideal, and   $\conv(I)$ be the convex hull of the set $\{\ab\in \RR^n\: \xb^{\ab}\in I\}$. Then  $\conv(I)$ is a polyhedron. Let $\{\ab_1\},\ldots,\{\ab_r\}$ be the  $0$-dimensional faces of $\conv(I)$. Then $J=(\xb^{\ab_1},\ldots, \xb^{\ab_r})$ is the unique minimal monomial reduction of $I$.
\end{Proposition}

For example, let  $I=(x^7,x^6y^2, x^3y^3,x^2y^5,xy^6,y^{10})$. Then $J=(x^7, x^3y^3,xy^6,y^{10})$ is the unique minimal monomial reduction of $I$.  Figure~\ref{better} demonstrates this theorem in our example.

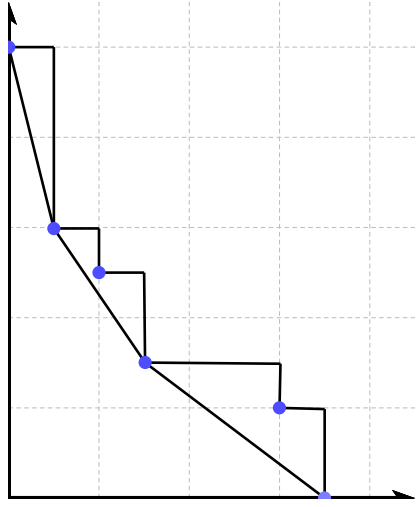
\begin{figure}[hbt]
\begin{center}
\newrgbcolor{xdxdff}{0.49019607843137253 0.49019607843137253 1.}
\newrgbcolor{ududff}{0.30196078431372547 0.30196078431372547 1.}
\psset{xunit=0.6cm,yunit=0.6cm,algebraic=true,dimen=middle,dotstyle=o,dotsize=5pt 0,linewidth=1.6pt,arrowsize=3pt 2,arrowinset=0.25}
\begin{pspicture*}(0.,0.)(9.,11.)
\multips(0,0)(0,2.0){6}{\psline[linestyle=dashed,linecap=1,dash=1.5pt 1.5pt,linewidth=0.4pt,linecolor=lightgray]{c-c}(0.,0)(9.,0)}
\multips(0,0)(2.0,0){5}{\psline[linestyle=dashed,linecap=1,dash=1.5pt 1.5pt,linewidth=0.4pt,linecolor=lightgray]{c-c}(0,0.)(0,11.)}
\psaxes[labelFontSize=\scriptstyle,xAxis=true,yAxis=true,Dx=2.,Dy=2.,ticksize=-2pt 0,subticks=2]{->}(0,0)(0.,0.)(9.,11.)
\psline[linewidth=1.pt](1.,5.9791270293466985)(2.,5.9791270293466985)
\psline[linewidth=1.pt](2.,5.)(2.,5.9791270293466985)
\psline[linewidth=1.pt](2.,5.)(3.,5.)
\psline[linewidth=1.pt](3.,5.)(3.0208729706533015,3.004970679523019)
\psline[linewidth=1.pt](3.0208729706533015,3.004970679523019)(6.0208729706533015,2.979127029346699)
\psline[linewidth=1.pt](6.0208729706533015,2.979127029346699)(6.,2.)
\psline[linewidth=1.pt](0.,10.)(1.,10.)
\psline[linewidth=1.pt](1.,10.)(1.,5.9791270293466985)
\psline[linewidth=1.pt](0.,10.)(1.,5.9791270293466985)
\psline[linewidth=1.pt](1.,5.9791270293466985)(3.0208729706533015,3.004970679523019)
\psline[linewidth=1.pt](3.0208729706533015,3.004970679523019)(7.,0.)
\psline[linewidth=1.pt](6.,2.)(7.,1.97415634982368)
\psline[linewidth=1.pt](7.,1.97415634982368)(7.,0.)
\begin{scriptsize}
\psdots[dotstyle=*,linecolor=xdxdff](7.,0.)
\psdots[dotstyle=*,linecolor=ududff](0.,10.)
\psdots[dotstyle=*,linecolor=ududff](6.,2.)
\psdots[dotstyle=*,linecolor=ududff](3.0208729706533015,3.004970679523019)
\psdots[dotstyle=*,linecolor=ududff](2.,5.)
\psdots[dotstyle=*,linecolor=ududff](1.,5.9791270293466985)
\end{scriptsize}
\end{pspicture*}
\end{center}
\caption{$\conv(I)$ for $I=(x^7,x^6y^2, x^3y^3,x^2y^5,xy^6,y^{10})$.}
\label{better}
\end{figure}

\medskip
Since the minimal monomial reduction ideal $J$ of a monomial ideal $I$ is uniquely determined, we simply write $r(I)$ for $r_J(I)$,  if consider the monomial reduction number of $I$.

Now let $I$  be a monomial ideal of height 2 in $K[x,y]$. It is customary to denote by $G(I)$ the unique minimal set of monomial generators of $I$. Since $\height(I)=2$, the ideal $I$ contains pure powers of $x$ and $y$.

\begin{Corollary}
\label{terriblenew}
Let $I\subset K[x,y]$ be a monomial ideal with $G(I)=\{x^a, y^b,\ldots\}$. Then $J=(x^a, y^b)$ is a minimal monomial  reduction ideal of $I$, if and only if for all $x^cy^d\in  G(I)$ one has $bc+ad\geq ab$.
\end{Corollary}

\begin{proof}
Let $L$ be the line   in $\RR^2$  passing through the points $(a,0)$ and $(0,b)$. Then  $L=\{(x,y)\:\; bx+ay=ab\}$. Let $H^+_{a,b}$ be the halfspace defined by $L$   which does not contain the origin $(0,0)$. Then $(c,d)\in H^+_{a,b}$ if and only if $bc+ad\geq ab$. Therefore, the set of $0$-dimensional faces of $\conv(I)$ is equal to  $\{(a,0), (0,b)\}$ if and only if $bc+ad\geq ab$ for all $x^cy^d\in G(I)$. Hence the desired conclusion follows from Proposition~\ref{pooja}.
\end{proof}

For any monomial $u\in K[x,y]$, $u=x^cy^d$, let
\[
\nu(u)=(bc+ad)/ab.
\]
Note that for any two monomials $u,v\in K[x,y]$ one has $\nu(uv)=\nu(u)+\nu(v)$.

\medskip
We denote by $\Ic_{a,b}$ the set of monomial ideals $I\subset K[x,y]$ with $x^a,y^b\in G(I)$ and $\nu(u)\geq 1$ for all $u\in G(I)$. By Corollary~\ref{terriblenew}, the monomial ideals $I\in \Ic_{a,b}$ are precisely  the monomial ideals in $K[x,y]$ for which $J=(x^a,y^b)$ is the (unique) minimal  monomial reduction ideal.

Furthermore, we denote by $\Ic_{a,b}^1\subset \Ic_{a,b}$ the set of monomial ideals $I\in \Ic_{a,b}$ with $\nu(u)=1$ for all $u\in G(I)$. Note that  $I\in \Ic_{a,b}^1$ if and only if $(c,d)\in L$ for all $x^cy^d\in G(I)$, where,  as in the proof of Corollary~\ref{terriblenew}, $L$ is the line  passing through the points $(a,0)$ and $(0,b)$.

If $a=b$, and $I\in \Ic_{a,b}^1$, then  $I$ is {\em equigenerated}, that is, all monomial generators of $I$ are of same degree, namely of degree $a$. In general, if $a$ and $b$ are not necessarily equal, and $I\in \Ic_{a,b}^1$, then $I$ is {\em quasi-equigenerated} in the sense that $\deg(u)=ab$ for all $u\in G(I)$ with respect to the non-standard grading $\deg(x)=b$ and $\deg(y)=a$.
For example, $(x^3,y^6, x^2y^4)\in \Ic_{3,6}$,  $(x^3,y^6, xy^4)\in \Ic_{3,6}^1$,  and
$(x^3,y^6, xy)\not \in \Ic_{3,6}$.

\medskip
Let $L_0\subset L$ be the line segment connecting $(a,0)$ and $(0, b)$, and let $g=\gcd(a,b)$. Then
\begin{eqnarray}
\label{hand}
C=\{(i\frac{a}{g}, b-i\frac{b}{g})\:\; i=0,\ldots,g\}
\end{eqnarray}
is the set of integer points on $L_0$. Thus, if $I\in \Ic_{a,b}^1$, then there exists a  unique  subset $A\subset [0,g]$ with $0,g\in A$ such that $I=I_A$, where
\[
I_A=(x^{i\frac{a}{g}}y^{b-i\frac{b}{g}}\:\; i\in A),
\]
and $\Ic_{a,b}^1=\{I_A\:\ A\subset [0,g] \text{ with } 0,g\in A\}$. Here, for two integers $c\leq d$,  we denote by $[c,d]$  the set of integers $\{i\:\; c\leq i\leq d\}$.

Let   $B\subset  [0, g]$ with  $ 0,g \in B$ and $I_B= (x^{i\frac{a}{g}}y^{ b-i\frac{b}{g}})_{i\in B}$, then   $I_A+I_B=I_{A\union B}$.

We set
\[
A+B=\{a+b\:\; a\in A,b\in B\},
\]
and define  $kA$ recursively by setting $1A=A$ and $kA=A+(k-1)A$. Then $I_AI_B=I_{A+B}$ and $I_A^k=I_{kA}$.

\medskip
There is a strong relationship between reduction  numbers of an ideal $I$  and algebraic invariants of the fiber cone of $I$.
For a graded  ideal $I$ in the polynomial ring $S=K[x_1,\ldots,x_n]$,   $K$ is a field,   the {\em fiber cone} of $I$ is defined to be the graded $K$-algebra
\[
F(I)=\Dirsum_{k\geq 0}I^k/\mm I^k,
\]
where $\mm=(x_1,\ldots,x_n)$ is the graded maximal ideal of $S$.

\begin{Lemma}
\label{stapelbed}
Let $I\in\Ic_{a.b}$ and $J=(x^a,y^b)$, and let $\bar{J}$ be the ideal generated by the elements
$x^a+\mm I$ and $y^b+\mm I$ in $F(I)_1$, and let $\mm_{F(I)}$ be the graded maximal ideal of $F(I)$. Then $\bar{J}$ is a reduction ideal of $\mm_{F(I)}$ and
\[
r(I)=r_{\bar{J}}(\mm_{F(I)}).
\]
\end{Lemma}

\begin{proof}
Let $r=r(I)$, then $I^{r+1}=JI^r$. Hence, $I^{r+1}/\mm I^{r+1}=(J+\mm I)/\mm I)(I^r/\mm I^r)$, which implies that   $\mm_{F(I)}^{r+1}=\bar{J}\mm_{F(I)}^r$.  This shows that $r(I)\geq r_{\bar{J}}(\mm_{F(I)})$.

Conversely, let $r= r_{\bar{J}}(\mm_{F(I)})$ and let $v+\mm I^{r+1}\in \mm_{F(I)}^{r+1}$, where $v\in I^{r+1}\setminus \mm I^{r+1}$  is a monomial.  By our assumption,   $v+\mm I^{r+1}=(u+\mm I)(w+\mm I^r)=uw+\mm I^{r+1}$ for some monomial  $w\in I^r$ and $u=x^a$ or $u=y^b$.  Thus,  $v-uw\in \mm I^{r+1}$. Suppose  $v-uw\neq 0$.  Then $v\in \mm I^{r+1}$, since $\mm I^{r+1}$ is a monomial ideal. This is a contradiction. So $v=uw\in JI^r$,  as desired.
\end{proof}

For a monomial ideal $I$ with $G(I)=\{u_1,\ldots,u_m\}$ and a positive integer $c$,  we let  $I^{[c]}$ be the monomial ideal with $G(I^{[c]})=\{u_1^c,\ldots,u_m^c\}$.

\begin{Lemma}
\label{deleted}
Let $I\in \Ic_{a,b}$, and $c>0$ be  an integer. Then $I^{[c]}\in\Ic_{ac,bc}$,  and
\[
r(I)=r(I^{[c]}).
\]
\end{Lemma}

\begin{proof}
The proof follows from  the fact  that for any two  monomial ideals $M,L\subset K[x,y]$ one has $(ML)^{[c]}=M^{[c]}L^{[c]}$.
\end{proof}

In \cite{Va} Vasconcelos gives an upper bound for the reduction number of the graded maximal ideal of a standard graded $K$-algebra $A$ in terms of the  arithmetic degree of $A$. In general, if $M$ is a finitely generated graded $A$-module, the {\em arithmetic degree}  of $ M $ is defined to be the number
\[
\arithdeg(M)=\sum_{P\in \Ass(M)}\mult_P(M)e(A/P).
\]
Here $e(M)$  denotes the multiplicity of  $M$, and $\mult_P(M)$ the length of $\Gamma_{PA_P}(M_P)$, where
\[
\Gamma_{PA_P}(M_P)=\{x\in M_P\:\; (P^kA_P)x=0 \text{  for some $k$}\}.
\]
Applied to our situation, the result of Vasconcelos together with Lemma~\ref{stapelbed} gives

\begin{Theorem}
\label{vasco}
Let $I\in \Ic_{a,b}$.  Then
$
r(I)< \arithdeg(F(I)).
$
\end{Theorem}

\begin{Remark}
\label{associativity}
{\em It follows from the associativity formula for multiplicities (cf. \cite[Corollary 4.7.8]{BH}) that $e(A)=\arithdeg(A)$,  if $\dim A/P=\dim A$ for all $P\in \Ass(A)$. This is for example the case if $A$ is Cohen--Macaulay or $A$ is a domain.}
\end{Remark}

\medskip
We call a graded  ideal  $I\subset K[x,y]$  {\em quasi-equigenerated},  if there exists a non-standard grading of $K[x,y]$ such that $I$ is generated  by homogeneous polynomials with respect to this grading. Thus any quasi-equigenerated monomial ideal, as defined before, is also a quasi-equigenerated graded ideal.

In the case that $I$ is a quasi-equigenerated graded ideal, say, $I=(f_1,\ldots,f_m)$ with $\deg f_i=d$ for all $i$ for suitable degrees   of the variables, then  one has
\[
F(I)\iso K[f_1,\ldots,f_m]\subset S.
\]
In particular, if  $I=I_A\subset K[x,y]$ with $\{0,g\}\subseteq A\subseteq [0,g]$, then  $$F(I)\iso K[\{x^{i\frac{a}{g}}y^{ b-i\frac{b}{g}}\:\; i\in A\}].$$


\section{An upper bound for the  monomial  reduction number of quasi-equigenerated  monomial ideals in $K[x,y]$}

Let $A$ be a finite number of integers. We denote by $\gcd(A)$ the greatest common divisor of the integers belonging to $A$. As a consequence of Theorem~\ref{vasco} we obtain

\begin{Theorem}
\label{cold}
Let $I\in\Ic_{a,b}^1$ with $G(I)=\{x^{i\frac{a}{g}}y^{ b-i\frac{b}{g}}\:\; i\in A\}$, where $\{0,g\}\subset A\subseteq [0,g]$ and  $g=\gcd(a,b)$.  Then $e(F(I))=g/\gcd(A)$. In particular,  $r(I)<g/\gcd(A)$.
\end{Theorem}

\begin{proof}
We first show that  $e(F(I))=g/\gcd(A)$.  Let  $\gcd(A)=t$. If $t>1$, then $I=(I')^{[t]}$, where $G(I')=\{x^{i'(a'/g')}
y^{(b')-i'(b'/g')}\:\; i'\in A'\}$ where $A'=\{i/t\:\; i\in A\}$, $a'=a/t$, $b'=b/t$ and $g'=g/t$.

Because $F(I)\iso F(I')$, we get $e(F(I))=e(F(I'))$ and   $\gcd(A')=1$. Suppose we have the desired result for $I'$. Then  $e(F(I'))=\gcd(a',b')$, and hence
\[
e(F(I))=e(F(I'))=\gcd(a',b')=\gcd(a,b)/\gcd(A).
\]
Thus  we may assume from the very  beginning that  $\gcd(A)=1$, and have then to show that $e(F(I))=g$.

Let $Q=Q(F(J))$ be the quotient field of $F(J)$. Since $J$ is a reduction ideal of $I$, it follows that $F(I)$ is a finitely generated $F(J)$-module. Therefore, $Q\tensor_{F(J)}F(I)\iso Q^e$, where $e=e(F(I))$, see for example \cite[Corollary 4.7.9]{BH}.

Let $L=(u_i\:\; i\in [0,g])$, where $u_i=x^{i\frac{a}{g}}y ^{b-i\frac{b}{g}}$. We claim that
\begin{eqnarray}
\label{early}
Q\tensor_{F(J)}F(I)= Q\tensor_{F(J)}F(L).
\end{eqnarray}
The inclusion $Q\tensor_{F(J)}F(I)\subseteq Q\tensor_{F(J)}F(L)$ is obvious, because $F(I)\subseteq F(L)$. For the other inclusion it is enough to show  that $u_j\in Q\tensor_{F(J)}F(I)$ for all $j\in [1,g-1]$.

Since $\gcd(A)=1$, there exist non-negative integers $r_i$ such that $\sum_{i\in A}r_i i\equiv  j(\mod g)$. Therefore, $\sum_{i\in A}r_ii=j+rg$ for some non-negative integer $r$.  Let $\bar{r}=\sum_{i\in A}r_i$. Then
\begin{eqnarray*}
\prod_{i\in A}u_i^{r_i}&=& x^{(\sum_{i\in A}r_ii)\frac{a}{g}}y^{\sum_{i\in A}r_i(b-i\frac{b}{g})}=x^{(j+rg)\frac{a}{g}}y^{\bar{r}b-(j+rg)\frac{b}{g}}
=u_j\frac{u_g^ru_0^{\bar{r}-1}}{u_0^r}.
\end{eqnarray*}
This show that $F(L)\subset Q\tensor_{F(J)}F(I)$, and hence $Q\tensor_{F(J)}F(L)\subseteq Q\tensor_{F(J)}F(I)$. This proves the claim.

It follows from (\ref{early}) that $e(F(I))=e(F(L))$. Let $I'$ be the monomial ideal with $G(I')=(x^a, x^{\frac{a}{g}} y^{b-\frac{b}{g}}, y^b)$.   Then, $I'=I_{A'}$ with $A'=\{0,1,g\}$. Since $\gcd(A')=1$, we get as before,  $e(F(I'))=e(F(L))$, and hence $e(F(I))=e(F(I'))$. Let $\varphi\: K[z_1,z_2,z_3]\to F(I')=K[x^a, x^{\frac{a}{g}} y^{b-\frac{b}{g}}, y^b]$ be the $K$-algebra homomorphism with $z_1\mapsto x^a$, $z_2\mapsto  x^{\frac{a}{g}} y^{b-\frac{b}{g}}$ and   $z_3\mapsto  y^b$. Then $\Ker(\varphi)=(z_2^g-z_1z_3^{g-1})$, so that  $F(I')\iso K[z_1,z_2,z_3]/(z_2^g-z_1z_3^{g-1})$. It follows that $e(F(I'))=g$, as desired.

Now we  apply Theorem~\ref{vasco}, and  obtain that $r(I)<\arithdeg(F(I))$. Since \\
$I\in\Ic_{a,b}^1$  it follows that $F(I)$ is a domain. Therefore, $\Ass(F(I))=\{(0)\}$,  and hence
\[
\arithdeg(F(I))=\mult_{(0)}(F(I))e(F(I))=e(F(I)).
\]
\end{proof}

The upper bound for the reduction number of a quasi-equigenerated  monomial ideal as given in Theorem~\ref{cold} can be improved by using a strong  result of Gruson, Lazarsfeld and Peskine  \cite{GLP}.

\begin{Theorem}
\label{GLP}
Let $I\subset K[x,y]$ be a quasi-equigenerated  monomial ideal with $G(I)=\{x^{i\frac{a}{g}}y^{ b-i\frac{b}{g}}\:\; i\in A\}$, where $\{0,g\}\subset A\subseteq [0,g]$ with  $g=\gcd(a,b)$. Then
\[
r(I)\leq \frac{g}{\gcd(A)}-|A|+2.
\]
\end{Theorem}

\begin{proof}
We may assume that $K=\CC$, since for monomial ideals the reduction number is independent of the base field. Then  the fiber cone of $I$ may be viewed as the homogeneous coordinate ring of an irreducible curve in $\PP^{|A|-1}$.  By \cite[Theorem 1.1]{GLP}, it follows that $\reg(F(I))\leq e(F(I))-|A|+2$. In the proof of Theorem~\ref{cold} we have seen that $e(F(I))=\frac{g}{\gcd(A)}$.  By a result of Trung~\cite{Tr},  one has $r_{\bar{J}}(\mm_{F(I)})\leq \reg F(I)$. Thus  the desired result follows from Lemma~\ref{stapelbed}.
\end{proof}

\section{On the values for the reduction number for quasi-equigenerated monomial ideals $I\subset K[x, y]$}
\label{value}

In this section we show for any   $0\leq j<\gcd(a,b)-1$, there exists  a quasi-equigenerated monomial ideal $I\subset K[x,y]$ with $x^a,y^b\in G(I)$ and $r(I)=j$. We also classify the quasi-equigenerated monomial ideals with smallest positive reduction number, namely reduction number 1,  and those with maximal reduction number.

\begin{Theorem}
\label{Somayeh}
Let  $I\in \Ic_{a,b}^1$. Let $g=\gcd(a,b)$ and $j\in [1,g-1]$. Then
$r(I)=j$, where $I=(x^{i\frac{a}{g}}y^{ b-i\frac{b}{g}}\:\;i\in A)$ with $A=[0,1] \union [j+1,g]$.
\end{Theorem}

\begin{proof}
 For the proof of the theorem we have to show that
$I^{j+1}=JI^j$ and $I^{j}\neq JI^{j-1}$, where $J=(x^a,y^b)$. Equivalently,
\begin{enumerate}
\item[(i)] $(j+1)A=\{0,g\}+ jA$, and
\item[(ii)] $jA\neq \{0,g\}+ (j-1)A$.
\end{enumerate}

Proof of (i): It is obvious  that $\{0,g\}+ jA\subseteq(j+1)A$. So it is enough to show that $(j+1)A\subseteq \{0,g\}+ jA$. In other words, we have to show: given $r_i\in A$ for $i=1,\ldots,j+1$, then there exist $r_1', \ldots r_j'\in A$ such that
\[
r:=r_1+r_2+\cdots +r_{j+1}=r_1'+\cdots +r_j'\quad \text{or}\quad r_1+r_2+\cdots +r_{j+1}=g+r_1'+\cdots +r_j'.
\]
If for some $i$ we have $r_i=0$ or $r_i=g$, then the assertion is trivial. Hence for the rest of the proof we may assume that $r_i\neq 0,g$ for all $i$.

\medskip
We consider different cases.

\medskip
\noindent
{\sc  Case 1}: There exist $k\neq \ell$ with  $r_k=1$ and $r_\ell\neq 1$. We may assume $r_1=1$ and $r_2\neq 1$. Then $j+1\leq r_2< g$, and it follows that
\[
r=1+r_2+\cdots +r_{j+1}=(1+r_2)+r_3+\cdots +r_{j+1}.
\]
Since $r_2<g$, it follows $r_2+1\in A$, so that $r$ is the sum of $j$  elements  belonging to~$A$.

\medskip
\noindent
{\sc  Case 2}: $r_i=1$ for all $i$. Then
\[
r=j+1= (j+1)+ \underbrace{0+\cdots+0}_\text{$j-1$}.
\]

\medskip
\noindent
{\sc Case 3}:  $r_i\neq 1$ for all $i$. Then $j+1\leq r_i<g$ for all $i$. Suppose  there exist $k\neq \ell$ with  $r_k+r_\ell\leq g$, say $k=1$ and $\ell=2$, then
\[
r= (r_1+r_2)+r_3+\cdots +r_{j+1},
\]
and we are done because $r_1+r_2\in A$. So in the sequel we may assume that $r_k+r_\ell> g$ for all $k\neq \ell$.

\medskip
\noindent
{\sc Subcase 3.1}: We assume $r_k+r_\ell \not\in[g+1,g+j]$ for some $k\neq \ell$. Then $(r_k+r_\ell)-g\in A$. We may assume $k=1$ and $\ell=2$, and then we get
\[
r= g+((r_1+r_2)-g)+r_3+\cdots +r_{j+1}.
\]
Since $(r_1+r_2)-g\in A$, we see that $r\in g+jA$.

\medskip
\noindent
{\sc Subcase 3.2}: We assume $r_k+r_\ell \in[g+1,g+j]$ for all $k\neq \ell$. If there exists $i$ such that $r_i\leq g-2$, we assume that $i=3$, and then
\[
r=(r_1+r_2-j)+(r_3+2)+(r_4+1)+\cdots +(r_{j+1}+1).
\]
Since all summands on the right hand side belong to $A$, we are done in this case.

If there exists no $i$ such that $r_i\leq g-2$, then $r_i=g-1$ for all $i$, and hence $2g-2=r_1+r_2 \leq g+j$. This implies that $g\leq j+2$. On the other hand $j+1\leq r_1=g-1$. Therefore, $g=j+2$, and
\[
r=(g-1)^2 =g + \underbrace{g+\cdots+g}_\text{$g-3$} +1.
\]
Since $g-3=j-1$,  it follows that $r\in g+jA$, and the  proof of (i) is completed.

\medskip
Proof of (ii): We claim that $j\in jA\setminus (\{0,g\}+(j-1)A)$. Indeed, $j\in jA$, because $[0,j]\subset jA$.
On the other hand, since $(j-1)A)=\Union_{i=0}^{j-1}(j-1-i)[0,1]+i[j+1,g]$, it follows  $\{0,g\}+(j-1)A\subseteq [0,j-1]\union [j+1,jg]$. Therefore,  $j\not\in  \{0,g\}+(j-1)A$.
\end{proof}

The next result classifies all quasi-equigenerated monomial ideal with reduction number 1.

\begin{Theorem}
\label{redone}
Let $I=I_A\in\Ic_{a,b}^1$ with $g=\gcd(a,b)>1$. Then
\[
r(I)=1\quad\text{if and only if}\quad \text{$A=\{0,d,2d,\ldots,(g/d)d= g\}$ and  $d=\gcd(A)\neq g.$}
\]
\end{Theorem}

\begin{proof}
Let $\gcd(A)=d$. Then $I_A=(I_{A'})^{[d]}$ with $\gcd(A')=1$, $I_{A'}\in \Ic_{a',b'}^1$ where $a'=a/d$, $b'=b/d$ and $g'=\gcd(a',b')=g/d$. Since $r(I)=r(I_{A'})$ by Lemma~\ref{deleted}, it suffices to show that $r(I_{A'})=1$ if and only if $A'=[0,g']$. Hence we may as well assume that $\gcd(A)=1$ and we show that $r(I)=1$ if and only if $A=[0,g]$.

Suppose first that  $A=[0,g].$ Then $I=(x^{ia/g}y^{b-ib/g})_{i=0, \ldots, g}=(x^{a/g},y^{b/g})^g$. Note that $$(x^{a/g},y^{b/g})^{2g}=(x^{a},y^{b})(x^{a/g},y^{b/g})^g.$$
 This shows that $r(I)=1$, since $g>1$.

Conversely, suppose  $r(I)=1$. Then $I^2=JI$, where $I=(x^{ia/g}y^{b-ib/g})_{i\in A}$ and $J=(x^{a},y^{b})$. This implies that $2A=\{0, g\}+A$. This last equation implies that $ 2A\equiv A (\mod g)$. Therefore,
\[
\bar{A}=\{i+g\ZZ \:\; i\in A\}
\]
is a subgroup of $\ZZ/(g)$. Since $\gcd(A)=1$, there exist $z_i\in \ZZ$ such that $1=\sum_{i\in A}z_ii$. Hence, $1+g\ZZ=\sum_{i\in A}z_i(i+g\ZZ)\in \bar {A}$, and so $\bar{A}=\ZZ/(g)$.  This implies that $A=[0,g].$
\end{proof}

Now we classify the quasi-equigenerated monomial ideals $I\subset K[x,y]$ with maximal reduction number.
First we show

\begin{Proposition}
\label{masiproves}
Let $a,b$ be positive integers, and let  $I=(x^a, x^{e(a/g)}y^{b-e(b/g)}, y^b)$,  where   $g=\gcd(a,b)$ and $e\in[1, g-1]$. Then $r(I)=(g/\gcd(e,g))-1$. \end{Proposition}
\begin{proof}
The ideal $I\subset K[x,y]$ is quasi-equigenerated  with $G(I)=\{x^{i\frac{a}{g}}y^{ b-i\frac{b}{g}}\:\; i\in A\}$, where $A=\{0,e,g\}$.  Then $I=I_{A}$.
Let $k=\gcd(e,g)$. It follows from Theorem~\ref{GLP}  that
$I^{g/k}=JI^{(g/k)-1}$, where $J=(x^a,y^b)$.  It remains to be shown that  $I^{(g/k)-1}\neq JI^{(g/k)-2}$, which means that $((g/k)-1)A\neq \{0,g\}+ ((g/k)-2)A$.

\medskip
Indeed, an arbitrary element of $(g/k)-2)A$ is of the form $r_1e+r_2g$, where $0\leq r_1+r_2\leq (g/k)-2$.
For any $0\leq j\leq (g/k)-2$, let
\[
A_j=\{je+(\ell-j)g:\ 0\leq j\leq\ell\leq (g/k)-2\}.
\]
One can see that $((g/k)-2)A=\bigcup_{j=0}^{(g/k)-2} A_j$.

By contradiction,  if $((g/k)-1)e\in ((g/k)-2)A$, then $((g/k)-1)e\in A_j$  and $((g/k)-1)e=je+(\ell-j)g$ for some $0\leq j\leq \ell \leq (g/k)-2$. Thus
$((g/k)-1)(e/k)-j(e/k)=(\ell-j)(g/k)$. So $g/k$  divides $((g/k)-1-j)(e/k)$. Since $\gcd(g/k,e/k)=1$, it follows that  $g/k$ divides $(g/k)-1-j$, which is a contradiction. By a similar argument,  if $((g/k)-1)e\in ((g/k)-2)A+g$, then $((g/k)-1)e=je+(\ell-j+1)g$ for some $0\leq j \leq \ell \leq (g/k)-2$,  and hence
$g/k$ divides  $(g/k)-1-j$, a contradiction. Thus $((g/k)-1)e\notin \{0,g\}+ ((g/k)-2)A$.
\end{proof}

Now we get

\begin{Theorem}
\label{sunshine}
Let $I\in \Ic_{a,b}^1$. Let $g=\gcd(a,b)$. Then $r(I)=g-1$,  if and only  if
$g=1$ and $I=(x^a,y^b)$, or
\[
I=(x^a, x^{e(a/g)}y^{b-e(b/g)}, y^b),
\]
where  $e\in [1,g-1]$ and $\gcd(e,g)=1$.
\end{Theorem}

\begin{proof}
 If we assume that $r(I)=g-1$, then Theorem~\ref{GLP} yields  that $g-1\leq \frac{g}{\gcd(A)}-|A|+2$. This implies that $|A|\leq 3$. If $|A|=2$, then $I=(x^a,y^b)$ and  $r(I)=0$, and hence $g=1$. If $|A|=3$, the result follows from Proposition~\ref{masiproves}.

Conversely, if $I=(x^a,y^b)$ with $\gcd(a,b)=1$, then clearly $r(I)=g-1$, and if $I=(x^a, x^{e(a/g)}y^{b-e(b/g),y^b})$ with  $c\in [1,g-1]$ and $\gcd(e,g)=1$,   then $r(I)=g-1$, again  by Proposition~\ref{masiproves}.
\end{proof}

Let $\Ic_{a,a}^1(j)$ be the set of all monomial ideals $I\in \Ic_{a,a}^1$ with $r(I)=j$. Theorem~\ref{Somayeh} implies that $\Ic_{a,a}^{1}(j)\neq \emptyset$ for $j=1,\ldots, a-1$. Let $m_a(j)=|\Ic_{a,a}^1(j)|$ and $m_a=|\Ic_{a,a}^1|$. Note that $m_a=2^{a-1}$.  It would be of interest to have  bounds for the ratio $m_a(j)/m_a$, and to know whether  $\lim_{a\to\infty} m_a(j)/m_a$ exists. Actually we expect that this limits exist and that $m(j):=\lim_{a\to\infty} m_a(j)/m_a=0$. For example, we have $m(1)=0$. Indeed, by  Theorem~\ref{redone} we obtain that $m_a(1)$ equals the number of  divisors of $a$ which are different from $a$. It follows that $1\leq m_a(1)<a$, and hence $0\leq m_a(1)/m_a<a/2^{a-1}$. This yields the desired conclusion.

We may ask similar questions when we restrict ourselves only to $3$-generated ideals. In this case we let $n_a$ be the number of all $3$-generated ideals of $\Ic_{a,a}^1$, and $n_{a}(j)$ be the number of all $3$-generated ideals of $\Ic_{a,a}^1$ of reduction number $j$. Here we can show the following

\begin{Proposition}
\label{ohohoh}
$\limsup_{a\to \infty}n_{a}(a-1)/n_a=1$ and $\liminf_{a\to \infty}n_a(a-1)/n_a=0$. In particular, $\lim_{a\to \infty}n_a(a-1)/n_a$ does not exist.
\end{Proposition}

\begin{proof}
It follows from Proposition~\ref{masiproves} that $n_a(d-1)=\phi(a)$, where $\phi$ is the well-known Euler function, while $n_a=a-1$. Hardy and Wright \cite[Theorem 326 ]{HW} showed that $\limsup_{a\to\infty}\phi(a)/a=1$. Also in Theorem~328 they showed that $$\liminf_{a\to\infty}(\phi(a)/a)\log \log a=e^{-\gamma},$$ where $\gamma = 0.577215665\ldots $ is the Euler constant. This implies that $\liminf_{a\to\infty}(\phi(a)/a)=0$. Therefore, we also have $\limsup_{a\to\infty}\phi(a)/(a-1)=1$ and $\liminf_{a\to\infty}\phi(a)/(a-1)=0$, as desired.
\end{proof}

\section{On the monomial reduction number for monomial ideals with $3$ generators}
\label{regensburg}

In this section we study the reduction number of ideals $I\in \Ic_{a,b}$ which are generated by 3 elements.  Let $I$ be such an ideal. Then
 \[
 I=(x^a,y^b, x^cy^d) \text{ and $0<c<a$, $0<d<b$  and $ad+bc\geq ab$.}
 \]
 Let $g=\gcd(a,b)$. It follows from Proposition~\ref{masiproves} that
 \begin{eqnarray}
 \label{veryfreshliver}
 r(I)=g/(\gcd(cg/a,g)-1, \quad \text{if}\quad  ad+bc= ab.
 \end{eqnarray}

\medskip
Now we consider the case that $I$ is not necessarily  quasi-equigenerated. Let, as before,  $$H^+_{a,b}=\{(c,d) \in \ZZ^2\:\;  bc+ad\geq ab\}.$$

We set $I_p=(x^a,y^b,x^cy^d)$ for $p =(c,d)\in H^+_{a,b}$, and $u_p=x^cy^d$.

\medskip
The following result characterizes the reduction  number of $I_p$.

\begin{Proposition}
\label{characterize}
With the assumptions and notation introduced we have
\[
r(I_p)=\min\{k \:   u_p^k\in J^k\}-1.
\]
\end{Proposition}

\begin{proof}
We first show that
\begin{eqnarray}
\label{evening}
G(I_p^k)\setminus G(JI_p^{k-1})\subseteq\{u_p^k\} \quad\text{for all}\quad k> 0.
\end{eqnarray}

Indeed,  $I_p=J+L$, where $L=(u_p)$, and   hence
\[
I_p^k=J^k+J^{k-1}L+\cdots + JL^{k-1}+L^k \quad\text{and}\quad JI_p^{k-1}=J^k+J^{k-1}L+\cdots + JL^{k-1}.
\]
Therefore, $I_p^k= JI_p^{k-1}+L^k$, which implies that $G(I_p^k) \subseteq G(JI_p^{k-1})\union G(L^k)$. Since  $G(L^k)=\{u_p^k\}$, the assertion follows.

Recall  that  $r(I_p)= \min\{k \:   I_p^k=JI_p^{k-1}\}-1$. Thus it follows from  (\ref{evening})  that $r(I_p)= \min\{k \:  u_p^k\in JI_p^{k-1}\}-1$.
Now if $u_p^k\in JI_p^{k-1}$, then $u_p^k\in J^{k-i}L^i$ for some $0\leq i\leq k-1$.
If $i>0$, then $u_p^{k-i}\in J^{k-i}\subseteq JI_p^{k-i-1}$. Hence in this case $k$ is not minimal with the property that  $u_p^k\in JI_p^{k-1}$. Therefore, $i=0$, and the desired result  follows.
\end{proof}

Note that $(\ZZ^2,<)$ is  a  partially  ordered set,  with
\[
(e,f)<(g,h) \Leftrightarrow\;  \text{$e<g$ and  $f\leq h$, or $e\leq g$ and  $f< h$}.
\]
The characterization of the reduction number given in Proposition~\ref{characterize} can be phrased as follows:
\medskip
Let $k$ be the smallest positive integer for which there exist positive   integers $k_1$ and $k_2$ such that
\begin{eqnarray}
\label{means}
k(c,d)\geq (k_1a,k_2b) \quad \text{and} \quad k=k_1+k_2.
\end{eqnarray}
Then $r(I_p)=k-1$.

\begin{Proposition}\label{I_p}
\label{goulash}
With the notation introduced we have
\begin{enumerate}
\item[(a)] $r(I_p)<\min\{a,b\}$  for all  $p\in  H^+_{a,b}$.
\item[(b)] $r(I_q)\leq r(I_p)$ for $p, q\in  H^+_{a,b}$ with $q\geq p$.
\end{enumerate}
\end{Proposition}

\begin{proof}




(a) We may assume $a\leq b$. We must show that $r(I_p)<a$. By Proposition~\ref{characterize},  $r(I_p)<a$,  if  $u_p^a\in J^a$.
This is the case if and only if  $ac\geq ia$ and $ad\geq(a-i)b$ for some $0\leq i \leq a$.  We may choose $i=c$. With $i=c$ the first inequality is trivially satisfied, while $ad\geq(a-c)b$  because $p\in  H^+_{a,b}$.

(b)  Let $p =(c,d)$ and $q=(e,f)$,  and let $r(I_p)=r_1+1$ and  $r(I_q)=r_2+1$. By (\ref{means}),  $r_1$ is smallest integer for which   $r_1c\geq ia$ and $r_1d\geq (r_1-i)b$ for some $i$. Since $e\geq c$ and $f\geq d$, it follows that   $r_1e\geq ia$ and $r_1f\geq (r_1-i)b$. This shows that $r_2\leq r_1$.
\end{proof}

The following corollary says that for any $a,b$,  the extremal values $1$ and $\min\{a,b\}-1$ for $r(I_p)$ are attained for suitable $p$.

\begin{Corollary}
\label{smallbig}
With the notation and assumptions as before, let $a\leq b$.Then
\[
r(I_p)=a-1,  \text{ if $p=(1,b-1)$, and }
\]
\[
r(I_p)=1,  \text{ if and only if $(a/2,b/2)\leq p=(c,d)$. }
\]
In particular, $r(I_p)=1$ for $p=(a-1,b-1)$.
\end{Corollary}
\begin{proof}
Let $p=(1,b-1)$. By Proposition \ref{goulash}, $r(I_p)\leq a-1$. We must show that for any integer $k<a$ we have
$(xy^{b-1})^k\notin(x^a,y^b)^k$, see Proposition~\ref{characterize}. Suppose that $(xy^{b-1})^k\in(x^a,y^b)^k$,  then exists  some integer $i$  with $0\leq i\leq k$ such that   $x^{(k-i)a}y^{ib} \mid x^ky^{(b-1)k}$.   It follows that $(k-i)a\leq k$ and $ib\leq (b-1)k$. The second inequality implies that  $i<k$,  and then the first inequality implies that $a\leq k$, a contradiction.

We have  $(a/2,b/2)\leq (c,d)$ if and only if $(a,b)\leq 2(c,d)$. Hence the assertion follows from (\ref{means}).
\end{proof}

Let $R_{(a,b)}=\{r(I_p)\: p\in  H^+_{a,b}\}$.  Then $R_{(a,b)}\subseteq [1,\min\{a,b\}-1]$. For example, $R_{(5,7)}= [1,4]$. In  general,  $R_{(a,b)}$ is a proper subset of $[1,\min\{a,b\}-1]$. For example,  $R_{(7,10)}\subsetneq [1,6]$, because $5\not \in R_{(7,10)}$.

It seems to be difficult to determine all pairs of positive integers $(a,b)$ for which $R_{(a,b)}=[1,\min\{a,b\}-1]$. However, in the next result we show that if $\min\{a,b\}>2$, then the set of reduction numbers $\{r(I_p)\:\; p\in L_0\}\neq [1,\min\{a,b\}-1]$.  Here $L_0$ is the line segment connecting $(a,0)$ with $(0,b)$. Indeed,  we have

\begin{Proposition}
\label{freshliver}
Let $a\leq b$ and $p\in L_0$. Then
\begin{enumerate}
\item[(a)]$r(I_p)< a/2$, if $a$ does not divide $b$.
\item[(b)] $r(I_p)\neq a-2$.
\end{enumerate}
\end{Proposition}

\begin{proof}
(a) Suppose $r(I_p)\geq  a/2$. Then (\ref{veryfreshliver}) implies that $a/2\leq g/\gcd(cg/a,g)-1$, where $g=\gcd(a,b)$. Since $g\leq a$ it follows that $\gcd(cg/a,g)=1$. This implies that $a/2\leq g-1$. Let $a=a'g$ with $a'$ a positive integer. Then $(a'g)/2<g$, and hence $a'/2<1$. Therefore, $a'=1$, and so $a=g$, a contradiction.

(b) Suppose $r(I_p)= a-2$. Then $a>2$ and $a-1=g/\gcd(cg/a,g)$, a contradiction because $a-1$ does not divide $a$.
\end{proof}

Let as before, $a\leq b$.  The next result shows that in general $(a-1)-|R_{(a,b)}|$ can be as big as we want. Indeed, we have

\begin{Proposition}
\label{ourlimits}
Let $p>2$ be a prime number. Then
\[
(p-1)-|R_{(p,p)}|\geq (p-1)/2-1.
\]
\end{Proposition}

\begin{proof}
Let $I_q=(x^p,y^p,x^cy^d)$ with  $q =(c,d)\in H^+_{a,b}$. By Proposition~\ref{masiproves}, for all integer points $q'$ on $L_0$  we know that $r(I_{q'})=p-1$. Next  we show that for all  $q'=(p-i,i+1)$ with $i=1,\ldots,p-2$ we have $r(I_{q'})\leq \frac{p-1}{2}$.
For the proof it is enough to show that there  exist  integers $j,k> 0$ with  $j+k\leq \frac{p-1}{2}$ such that either

\begin{enumerate}
\item[(i)]$(\frac{p+1}{2}) (p-i)\geq (p-i)[(\frac{p-1}{2})-(j+k)]+pj+p$, and
\item[(ii)] $(\frac{p+1}{2}) (i+1)\geq (i+1)[(\frac{p-1}{2})-(j+k)]+pk$
 \end{enumerate}
 or
\begin{enumerate}
\item[(iii)]$(\frac{p+1}{2}) (p-i)\geq (p-i)[(\frac{p-1}{2})-(j+k)]+pj$, and
\item[(iv)] $(\frac{p+1}{2}) (i+1)\geq (i+1)[(\frac{p-1}{2})-(j+k)]+pk+p$.
 \end{enumerate}
Indeed we should find $j$ and $k$ such that

\begin{enumerate}
\item[(i)$'$]$(p-i)(j+k+1)\geq p(j+1)$, and
\item[(ii)$'$] $((i+1)(j+k+1)\geq pk$
 \end{enumerate}
 or
\begin{enumerate}
\item[(iii)$'$]$(p-i)(j+k+1)\geq pj$, and
\item[(iv)$'$] $((i+1)(j+k+1)\geq p(k+1)$.
 \end{enumerate}
Suppose first that  $i$ is an even number. Then we choose $j= \frac{p-i-1}{2}$, $k=(i/2)-1$. Then (iii)$'$ and (iv)$'$ hold. Now suppose that $i$ is odd. Set $j= (\frac{p-i}{2})-1$ and $k=\frac{i+1}{2}$. Then the inequalities (i)$'$ and  (ii)$'$ hold for this choice of $j$ and $k$.

Now let $q''\in  H^+_{a,b}\setminus L_0$ arbitrary. Then there exists $q'=(p-i,i+1)$ with $q'\leq q''$. Together with  Proposition~\ref{goulash}, we get  $r(I_{q''})\leq r(I_q')\leq  \frac{p-1}{2}$. Therefore, $|R_{(p,p)}|\leq (p-1)/2+1$. This yields the desired conclusion.
\end{proof}

In contrast to  the previous result we have

\begin{Proposition}
\label{specialnight}
Let $a\geq 2$ be an integer. Then $$\bigcup_b R_{(a,b)}=\{1,\ldots,a-1\},$$
where the union is taken over all $b $ with $b\geq a$.
\end{Proposition}

\begin{proof}
For any $a$ and $b$  we know that  $R_{(a,b)}\subseteq [1,\min\{a,b\}-1]$. Hence, obviously one has
$\cup_{ b} R_{(a,b)}\subseteq\{1,\ldots,a-1\}.$  In order to prove the other inclusion, we apply (\ref{means}) and  must show that for given $a\geq 2$ and $r$ with $1\leq r-1\leq a-1$,  there exist $b,c,d$  with $b\geq a$ and $c<a,d<b$  such that $cr\geq at$,  $dr\geq bs$ for some positive integers $s$ and $t$ with   $r=s+t$.

We choose $t=r-1$, then $s=1$. With this choice of $s,t$ we may choose  $c=a-1$, $d=\lceil\frac{a}{r}\rceil$ and $ b=\lceil\frac{a}{r}\rceil r$. Then all required  inequalities are satisfied.   It is sufficient that we show $ r$ is a smallest number  for this  choice of $c,b$ and $d$. Now let   $r'$ be any positive integer  such that
 \begin{enumerate}
\item[(i)]$(a-1)r'\geq a(r'-i)$, and
\item[(ii)] $\lceil\frac{a}{r}\rceil r'\geq (\lceil\frac{a}{r}\rceil r)i$.
 \end{enumerate}
(ii) implies that  $r'\geq ri$,  and by  (i) we obtain that   $i\ne 0$.  Hence, $r'\ge r$, as desired.
\end{proof}

\section{On the   reduction numbers of a monomial ideal and  its  quasi-equigenerated part}
\label{part}

Let $I\in \Ic_{a,b}$.  By the definition of $\Ic_{a,b}$, we  have $\nu(u)\geq 1$ for all $u\in I$, where $\nu(u)=(ad+bc)/ab$, see Section~\ref{preliminaries}.  The  {\em quasi-equigenerated part} $I_0\in \Ic_{a,b}^1$ of $I$  is defined to be the monomial ideal generated by all $u\in G(I)$ with $\nu(u)=1$.

\medskip
In Proposition~\ref{goulash} it is shown that $r(I_q)\leq r(I_p)$ if $I_q\subseteq I_p$. One may wonder, whether  always $r(I)\leq r(I')$, if $I\subseteq I'$. Simple examples show that is not always the  case. However, we have
\begin{Proposition}
\label{turtle}
Let $I\in \Ic_{a,b}$ and $I_0$ be the quasi-equigenerated part of $I$.  Then
\[
G(I_0^k)\setminus G(JI_0^{k-1})\subseteq G(I^k)\setminus G(JI^{k-1}).
\]
In particular, $r(I_0)\leq r(I)$.
\end{Proposition}
\begin{proof}
Let $u\in G(I_0^k)$ and suppose that $u\in G(JI^{k-1})$. We have to show that $u\in G(JI_0^{k-1})$. Assume $u\notin G(JI_0^{k-1})$. We can write $I=I_0+I_1$, where $I_1$ is generated by the monomials $x^cy^d\in G(I)$ such that $bc+ad> ab$, then
\[
JI^{k-1}=J(I_0+I_1)^{k-1}=JI_0^{k-1}+ JI_0^{k-2}I_1+\cdots +JI_1^{k-1}.
\]
This implies that $G(JI^{k-1})\subseteq \Union_{i=0}^{k-1} G(JI_0^{k-i-1}I_1^i)$. Thus it suffices to show that $u\notin G(JI_0^{k-i}I_1^i)$, for all $i\geq 1$.

Since $u\in G(I_0^k)$, it follows that $\nu(u)=k$. Suppose there exist some $i\geq 1$ such that $u\in JI_0^{k-i}I_1^i$. Then, by using again  (i) and (ii), it follows that $\nu(u)>k$, a contradiction.
\end{proof}

In the following special case,  the reduction number of $I$ is determined by  the reduction number of its quasi-equigenerated part.

\begin{Theorem}
\label{onion}
Let $1\leq a\leq b$ be integers and $g=\gcd(a,b)$. Furthermore, let  $A\subset [0,g]$ with $\{0,g\}\subseteq A$.  Let $r\in \RR,\;  r\geq 1$ and  $I=I_A+I'$, where $I'=(u\:\;  \nu(u)>r)$.

\medskip
Then $I_A$ is the quasi-equigenerated part of $I$,  and
\begin{enumerate}
\item[(a)] $r(I)=1$,  if \;   $r(I_A)=0$;

\item[(b)] $r(I)= r(I_A)$,  if\;  $r(I_A)>0$.
\end{enumerate}
In particular, $r(I)\leq g/\gcd(A)-|A|+2$.
\end{Theorem}
\begin{proof}
 (a)   We must show  $I^2 \subseteq JI$. We know that $I^2=(I_A+I')^2= (I_A)^2+I_AI' +I'^2$ and $JI=JI_A+JI'$.  The condition $r(I_A)=0$ means that $I_A=J$. Hence it is sufficient to show $I'^2\subseteq JI'$.  Let $u\in I'^2$. Then  $\nu (u)>2r\geq 2$.  Hence if $u=x^cy^d$, then $c\geq a$ or $d\geq b$. We may assume $c\geq a$. Then $u=x^au'$, where $u'$ is a monomial with $\nu(u')>2r-1\geq r$. Therefore $u'\in I'$ and $u\in JI'$.

 (b) With notation of Proposition~\ref{turtle} we have $I_0=I_A$ and $r(I_0)\leq r(I)$. So it is sufficient to show that $r(I_0)\geq r(I)$.
Suppose that $r(I_0)=m>0$. We want to show $I^{m+1}\subseteq JI^m$. Indeed, let $w\in I^{m+1}$. Then $w=u_1\cdots u_kv_1\cdots v_{m+1-k}$ with $u_i\in I_0$ and $v_j\in I'$. If $k=m+1$, then the assertion follows, since $I_0^{m+1}=JI_0^m$. Now let $k<m+1$. Since $\nu(u_kv_1)>1+r\geq 2$,  it follows from the proof of part (a), that $u_kv_1=w_1w_2$, where $w_1\in J$ and $w_2\in I'$. Therefore,  $w=w_1(u_1\cdots u_{k-1}w_2v_2\cdots v_{m+1-k})$ belongs to $JI^m$, as desired.
\end{proof}

Observe that the ideals in Theorem~\ref{onion} are lex ideals if $b>a$. Recall that a monomial ideal $I$ is called a {\em lex ideal}, if and only if for  all monomial  $u\in I$ and  all monomials $v$ with $\deg(v)=\deg(u)$ for which $v>u$ with respect to the lexicographical order, it follows that $v\in I$.

Indeed, let $I$ be as in  Theorem~\ref{onion} with $r=1$, and let  $u\in I$ and a monomial $v$ with $\deg(v)=\deg(u)$ and $v>u$.  Then  $\nu(v)> \nu(u)$, because $b>a$. This shows that $v\in I$, and hence $I$ is a lex ideal. In particular,   Theorem~\ref{onion} shows that  lex ideals may have any  reduction number $\geq 1$.

\section{Monomial ideals of reduction number $1$}
\label{unique}

In Section~\ref{value} we classified the quasi-equigenerated monomial ideals with  reduction number $1$. Here we consider more generally  monomial ideals   $I\in \Ic_{a,b}$   with $1\leq a\leq b$, and analyze what it means that $r(I)=1$. We have the following simple observation.

\begin{Lemma}
\label{obvious}
One has $r(I)=1$,  if and only if $I\neq J$,  and  for all monomials  $u_p,u_q\in I\setminus J$ with $u_pu_q=x^cy^d$ and $(c,d)\not\geq  (a,b)$, we have $x^{c-a}y^d\in I$ if $c\geq a$, or  $x^cy^{d-b}\in I$ if $d\geq b$.
\end{Lemma}

\begin{proof}
We have $r(I)=1$, if and only if $I^2=JI$. This is the case if and only if for all monomials  $u_p,u_q\in I\setminus J$, it follows that $u_pu_q\in JI$.

If $p+q\geq (a,b)$, then $u_pu_q\in J^2\subset JI$.  Suppose now that $p+q\not\geq (a,b)$. Then $c\geq a$ or $d\geq b$, because $cb+ad\geq 2ab$. Assume that $c\geq a$. Then $d<b$, since $p+q\not\geq (a,b)$, and  $u_pu_q=x^au$ where  $u=x^{c-a}y^d$. Hence $u_pu_q\in JI$  if and only if $u\in I$.
\end{proof}

\begin{Corollary}
\label{minimal}
Let $I\in \Ic_{a,b}$ with  $1<a\leq b$.  We assume that $I\neq (x^a,y^b)$.
\begin{enumerate}
\item[(a)]  Let $L_1, L_2\subset \Ic_{a,b}$ be monomial ideals such that  $I\subseteq L_1, L_2$ and $r(L_1)=r(L_2)=1$. Then $r(L_1\sect L_2)=1$.
\item[(b)] There exists a unique monomial ideal $L\subset \Ic_{a,b}$  with $I\subseteq L$ and $r(L)=1$ and such that $r(L')>1$ for all $L'$ with
$I\subseteq L'\subset L$.
\end{enumerate}
\end{Corollary}

\begin{proof}
(a) Let $J=(x^a,y^b)$.   Since $I\subseteq L_1, L_2$ it follows that $J\neq L_1, L_2$. Now let $u_p,u_q\in (L_1\sect L_2)\setminus J$ with $p+q=(c,d)\not\geq  (a,b)$. Then $u_p, u_q\in L_1\setminus J$. Since $r(L_1)=1$, Lemma~\ref{obvious} implies that $x^{c-a}y^d\in L_1$ if $c\geq a$, or  $x^cy^{d-b}\in L_1$ if $d\geq b$.  We may assume that $c\geq a$. Then $x^{c-a}y^d\in L_1$. By the same reason, since $r(L_2)=1$, we have $x^{c-a}y^d\in L_2$. Therefore, $x^{c-a}y^d\in L_1\sect L_2$, Similarly, if  $d\geq b$, we see that $x^cy^{d-b}\in L_1\sect L_2$. Therefore $r(L_1\sect L_2)=1$, by  Lemma~\ref{obvious}.

(b) By Theorem \ref{redone}, there exists a monomial ideal $L\supseteq I$ with
$r(L) = 1$. Let $\mathcal{L}$ be the set of all monomial ideals $L$ with $I\subseteq L$ and $r(L)=1$. By (a),   $L_0=\Sect_{L\in \mathcal{L}}L$ has reduction number $1$, and of course it is the unique smallest ideal with this property containing $I$.
\end{proof}

\begin{Examples}
\label{telegram}
{\em
(a) Let $I= (x^a,y^b, \{u_p\:\; p\in D\})$, where $D\subseteq \{p\:\; p\geq (a/2,b/2)\}$. Then $r(I)=1$. Indeed, $p+q\geq (a,b)$ for all $p,q$, and the assertion follows by Lemma~\ref{obvious}.

(b) Let $I_1=(x^4,y^8,x^3y^3)$ with $u_{p_1}=x^3y^3$. It can be seen that $\nu(u_{p_1})\geq 1$ and $r(I_1)=2$. We want to build the smallest  monomial ideal with reduction number $1$ containing $I$. At first, let $u_2= u_{p_1}^2=x^6y^6\in I_1^2$. We observe that $u_2=x^4(x^2y^6)$. Therefore, any ideal of  reduction number $1$ containing $I_1$ must contain $x^2y^6$.  Now consider $I_2=(x^4,y^8,u_{p_1}, u_{p_2})$,  where $u_{p_2}=x^2y^6$.  Then it can be checked that  $r(I_2)=1$, as desired.
}
\end{Examples}
In the next theorem we describe for any $3$-generated monomial ideal $I\subset K[x,y]$ for which $(x^a,y^b)$ is the monomial reduction ideal,  the unique smallest  monomial ideal $L$ with $I\subseteq L$ and $r(L)=1$.

\medskip
Let $1< a\leq b$ be integers. We  let $D_{a,b}=H_{a,b}^+\sect \{(i,j)\:\; i<a,\; j<b\}$.

\begin{Theorem}
\label{wetransfer}
 Let $1< a\leq b$ be integers,   and let  $p=(c,d)\in D_{a,b}$ and  $I=(x^a,y^b, u_p)$.  For all integers $i\geq 1$ we write
\[
ic=r_ia+c_i,  \quad id=s_ib+d_i,
\]
with $r_i,s_i,c_i,d_i$ integers such that   $0\leq c_i<a$  and $0\leq d_i<b$. We set $p_i=(c_i,d_i)$ for all $i\geq 1$. Let $k$  be the smallest integer such that $r_k+s_k\geq k$. Then $L=(x^a,y^b, u_{p_1}, \ldots,u_{p_{k-1}})$ is the unique  smallest monomial ideal containing $I$ such that $r(L)=1$.
\end{Theorem}

\begin{proof}
In the first step we  show that $r(L)=1$.

First we claim that
\begin{eqnarray}
\label{i-1}
r_i+s_i\geq i-1 \quad \text{for all}\quad i\geq 1.
\end{eqnarray}
For proof of claim, we consider $ibc=r_iab+c_ib$ and  $iad=s_iba+d_ia$. Now, since $u_p\in D_{a,b}$, we see that  $iab\leq i(cb+ad)=(r_i+s_i)ab+(c_ib+d_ia)$.  Because $0\leq c_i< a$ and $0\leq d_i< b$,   it follows that $(r_i+s_i)ab\geq iab-(c_ib+d_ia)>(i-2)ab$. It means $r_i+s_i\geq i-1$,  and the  claim is proved.

Next we show that $p_i\in D_{a,b}$ for $i=1,\ldots,k-1$. Indeed, for $i=1,\ldots,k-1$ we have $r_i+s_i=i-1$, by (\ref{i-1}) and the definition of $k$. Therefore, the equations $ic=r_ia+c_i$ and $id=s_ib+d_i$ imply that $i(cb+ad)=(i-1)ab+c_ib+d_ia$. Since $p=(c,d)\in D_{a,b}$,  it follows that $cb+ad\geq ab$. Hence
$c_ib+d_ia= i(cb+ad)-(i-1)ab\geq ab$. Since $c_i<a$ and $d_i<b$, we see that $p_i\in D_{a,b}$, as desired.

Now we prove that $u_{p_i}u_{p_j}\in JL$  for $i,j\leq k-1$. This then shows that $r(L)=1$.

We have
\[ (i+j)c=(r_i+r_j)a+(c_i+c_j),\quad \text{and} \quad (i+j)d=(s_i+s_j)b+(d_i+d_j).
\]
By Lemma~\ref{obvious} we may assume $(c_i+c_j, d_i+d_j)\not\geq (a,b)$. Then   $c_i+c_j\geq a $ and $d_i+d_j< b$,  or $c_i+c_j< a $ and $d_i+d_j\geq  b$, because $p_i\in D_{a,b}$ for $i=1,\ldots.k-1$. Let us  assume that $c_i+c_j\geq a $ and $d_i+d_j< b$.
Because $c_i+c_j\geq a $ we can write $(i+j)c=(r_i+r_j+1)a+(c_i+c_j-a)$ with $0\leq c_i+c_j-a<a$.  Therefore, $p_{i+j}= (c_i+c_j-a,  d_i+d_j)$, and $u_{p_i}u_{p_j}=x^au_{p_{i+j}}$. So,  if $i+j<k$, then $u_{p_{i+j}}\in L$, and hence $u_{p_i}u_{p_j}\in JL$.

Now let $i+j\ge k$. By  (\ref{i-1}) and the definition of $k$,  it follows $r_i+s_i=i-1$ and $r_j+s_j=j-1$.  and hence $r_i+r_j+1+s_i+s_j=i+j-1$.  Thus it suffice to prove the following claim
\begin{eqnarray*}
\text{$(*)$ Let $h\geq k$ and $r_h+s_h=h-1$. Then $u_{p_h}\in L$}.
\end{eqnarray*}

Let $h=k+l$. Since $r_k+s_k\geq k$, it follows that $l>0$. we have

\begin{eqnarray}
\label{summer}
  h c =(r_k+r_l)a+(c_k+c_l),
 \quad    hd=(s_k+s_l)b+(d_k+d_l),
\end{eqnarray}
and on the other hand
\begin{eqnarray}
\label{water}
hc=r_ha+c_h, \quad  hd=s_hb+d_h.
\end{eqnarray}
Comparing (\ref{summer}) with (\ref{water}), we see that
 \begin{eqnarray*}
\label{d}
r_k+r_l\leq r_h \quad \text{and}\quad s_k+s_l\leq s_h.
\end{eqnarray*}

Therefore,
\begin{eqnarray*}
\label{equall}
h-1=r_h+s_h\geq (r_k+s_k)+(r_l+s_l)\geq k+l-1=h-1,
\end{eqnarray*}
because $r_k+s_k\geq k$ and $r_l+s_l\geq l-1$,  due to our assumption on $k$  and due to (\ref{i-1}).  This  implies that
\begin{eqnarray}
\label{goood}
r_k+s_k= k \quad \text{and}\quad r_l+s_l= l-1.
\end{eqnarray}
We also observe that
\begin{eqnarray}
\label{verygoood}
u_{p_h}=u_{p_k}u_{p_l}.
\end{eqnarray}
Indeed, $u_{p_k}u_{p_l}=x^{c_k+c_l}y^{d_k+d_l}$ and $u_{p_h}=x^{c_h}y^{d_h}$. We have $c_h= c_k+c_l$ and $d_h=d_k+d_l$, if and only if $c_k+c_l<a$
and $d_k+d_l<b$. Suppose $c_k+c_l\geq a$ or  $d_k+d_l\geq b$. We may assume that $c_k+c_l\geq a$. Then
\[
hc=(r_k+r_l+1)a+c_k+c_l-a.
\]
As before, $r_k+r_l+1\leq r_h$ and $s_r+s_l\leq s_h$. Therefore,
\[
h=k+l=(r_k+r_l+1)+(s_r+s_l)\leq r_h+s_h=h-1,
\]
a contradiction.

Let $\{h_1,h_2,\ldots, \}$ be the set of integers $>k$ with the property that $h_1<h_2<\cdots $ and  $r_{h_i}+s_{h_i}=h_i-1$ for all $i$.

Now we prove $(*)$ by induction on $i$. Let $h_1=k+l_1$. Then $r_{l_1}+s_{l_1}= l_1-1$, by   (\ref{goood}). Since $h_1$ is the smallest element $\geq k$ for which such an equation holds, it follows that $l_1<k$. Therefore,  $u_{p_{l_1}}\in L$. Thus (\ref{verygoood}) implies that  $u_{p_{h_1}}\in L$

Now let $i>1$, then $l_i<k$  or $l_i>k$. If $l_i<k$, then as before we have that  $u_{p_{h_i}}\in L$, and if $l_i>k$, then there exists $j<i$ such that $h_j=l_i$ because $l_i<h_i$ and $r_{l_i}+s_{l_i}=l_i-1$. By induction hypothesis it follows that $u_{p_{l_i}}=u_{p_{h_j}}\in L$. Again by using (\ref{verygoood}), we conclude that $u_{p_{h_i}}\in L$.

\medskip
In the second step we show that $L$ is the smallest monomial ideal with $r(L)=1$  containing $I$. Let $L'$ be  the unique smallest monomial ideal with $r(L')=1$. Then $L'\subseteq L$. By induction on $i$ we show that $u_{p_i}\in L'$ for $i=1,\ldots,k-1$. Then $L\subseteq L'$, and hence we have equality.
The induction begin is trivial, because $p_1=p$ and $u_p\in I\subseteq L'$. Now let $1<i\leq k-1$. By induction hypothesis, $u_{p_{i-1}}\in L'$. Then
$u_{p}u_{p_{i-1}}\in (L')^2$.  We have $u_{p}u_{p_{i-1}}=x^{c_1+c_{i-1}}y^{d_1+d_{i-1}}$. Suppose $(c_1+c_{i-1}, d_1+d_{i-1})\geq (a,b).$ Then adding the equations $ic=r_{i-1}a+c_1+c_{i-1}$ and $id=s_{i-1}b+d_1+d_{i-1}$, we get
\[
i(c+d)=(i-2)(a+b)+(c_1+c_{i-1})+ (d_1+d_{i-1})\geq (i-2)(a+b)+(a+b)=(i-1)(a+b).
\]
This implies that $i(c-a)+i(d-b)\leq -a-b$, a contradiction.  Therefore, $(c_1+c_{i-1}, d_1+d_{i-1})\not\geq (a,b).$ Then as shown in the first step we have $c_1+c_{i-1}\geq a $ and $d_1+d_{i-1}< b$,  or $c_1+c_{i-1}< a $ and $d_1+d_{i-1}\geq  b$. We may assume  $c_1+c_{i-1}\geq a $ and $d_1+d_{i-1}< b$. Then, as above we get $u_{p}u_{p_{i-1}}=x^au_{p_{i}}$. So $x^au_{p_{i}}\in (L')^2$. Since $r(L')=1$, we have  $x^au_{p_{i}}\in JL'$, and  since $y^b$ does not divide $x^au_{p_{i}}$ it follows that $u_{p_{i}}\in L'$.
\end{proof}

\section{Monomial reductions and powers}
\label{section}

In this part we study the reduction numbers of powers of ideals  which belong to $\Ic_{a,b}$.

The following inequalities  are an immediate consequence of a more general result due to Hoa \cite[Lemma 2.7]{Hoa},  applied to our situation.

\begin{Proposition}
\label{hoa}
 Let $I\in \Ic_{a,b}$.  Then  $r(I^k)\leq \lceil (r(I)-1)/k\rceil +1$.
\end{Proposition}

These inequalities imply that $r(I^k)\leq r(I)$ for all $k$. We even expect that
\begin{eqnarray}
\label{weare}
r(I^{k+1})\leq r(I^k)\quad \text{for all} \quad k\geq 1.
\end{eqnarray}
The inequalities (\ref{weare}) imply in particular that if $r(I^{k_0})=1$, then $r(I^k)=1$ for $k\geq k_0$. For $I\in \Ic_{a,b}^1$ we can show this without using \eqref{weare}.

\begin{Theorem}
\label{related}
Let $I\in \Ic_{a,b}^1$, where $I=I_A$ . Furthermore,  let $g=\gcd(a,b)$.
\begin{enumerate}
\item[(i)] Let $\gcd(A)=1$. Then the following conditions are equivalent:
\begin{enumerate}
\item[(a)] $r(I^k)=1$ for some $k$.
\item[(b)] $\{0,1,g-1,g\}\subseteq A$.
\item[(c)] $r(I^k)=1$  for all  $k\geq g-2$.
\end{enumerate}
\item[(ii)]  If $r(I^{k_0})=1$, then $r(I^k)=1$ for $k\geq k_0$.
\end{enumerate}
\end{Theorem}

\begin{proof}
(i): Note that $I^k =I_{kA}$. Since $\gcd(A)=1$, we have $\gcd(kA)=1$. Therefore,  Theorem~\ref{redone} implies that
\begin{eqnarray}
\label{trivial}
r(I^k)=1 \quad \Longleftrightarrow \quad kA=[0,kg].
\end{eqnarray}
 (a)\implies (b): Let $r(I^k)=1$, and assume that $1\not\in A$. Then $1\not\in kA$, contradicting (\ref{trivial}). By symmetry it also follows that $g-1\in A$.

 (b)\implies (c): By (\ref{trivial}), we have to show that  $kA=[0,kg]$ for $k\geq g-2$. Let $B=\{0,1,g-1,g\}$. By assumption, $B\subseteq A$. Observe that
\begin{eqnarray}
\label{masoomehsays}
 kB=\Union_{i=0}^k[i(g-1),i(g-1)+k] \quad\text{for all}\quad k.
\end{eqnarray}
We prove (\ref{masoomehsays}) by induction on $k$. For $k=1$, the assertion is trivial. Now assume that $kB=\Union_{i=0}^k[i(g-1),i(g-1)+k]$. Then
\begin{eqnarray*}
&&(k+1)B=kB+B=(kB+[0,1])\union (kB+[g-1,g])\\
&=& \Union_{i=0}^k[i(g-1),i(g-1)+k+1]
\union  \Union_{i=0}^k[(i+1)(g-1),(i+1)(g-1)+k+1]\\
&=& \Union_{i=0}^{k+1}[i(g-1),i(g-1)+k+1].
\end{eqnarray*}

In particular, (\ref{masoomehsays}) implies that  $kB=[0,kg]$ for $k\geq g-2$.  Therefore, $[0,kg]=kB\subseteq kA\subseteq [0,kg]$  for $k\geq g-2$, and hence $kA= [0,kg]$ for $k\geq g-2$.

(c)\implies (a) is trivial.

(ii): Let $d=\gcd(A)$. Then $I=(I_{A'})^{[d]}$ with $\gcd(A')=1$ and $I^k=(I_{A'}^k)^{[d]}$. Therefore, by Lemma~\ref{deleted} we have $r(I^k)=r(I_{A'}^k)$ for all $k$. Hence we may assume that $\gcd(A)=1$ and it is enough to show that $r(I^{k+1})=1$ if $r(I^k)=1$. Indeed, $kA=[0,kg]$, and $(k+1)A=kA+A\supseteq kA+\{0,1,g-1,g\}$, because of (i). Hence
\begin{eqnarray*}
(k+1)A&\supseteq &kA+([0,1]\union [g-1,g])= ([0,kg]+[0,1])\union ([0,kg]+[g-1,g])\\
&=&[0,kg+1]\union [g-1,(k+1)g]= [0,(k+1)g],
\end{eqnarray*}
since $g-1\leq kg+1$. Thus (\ref{trivial}) implies that  $r(I^{k+1})=1$.
\end{proof}

The upper bound for the reduction number of $I$ given in Theorem~\ref{cold}  can be improved by $1$ under  additional assumptions.

\begin{Corollary}
\label{trickyteacher}
Let $I\in \Ic_{a,a}^1$, $I=I_A$ with   $\gcd(A)=1$.   Suppose that $r(I^k)=1$ for some $k$.  Then $r(I)\leq a-2$.
\end{Corollary}

\begin{proof}
Let $F(I)$ be the fiber cone of $I$. Since $I$ is equigenerated it follows that $F(I)\iso K[\{u\:\; u\in G(I)\}]$.
 Since $g=a$,  Theorem~\ref{related} implies that  $r(I^k)=1$ for some $k$ if and only if $\{0,1,a-1,a\}\subset A$.  Thus, $G(I)=\{f_1x,f_1y,g_1,\ldots,g_r, f_2x, f_2y\}$,  where  $f_1=x^{a-1}$, $f_2=y^{a-1}$ and the $g_i$ are monomials of degree $a$, Therefore, the hypotheses of  \cite[Theorem 1.1]{Hi} are satisfied, and we get $\reg(F(I))\leq a-2$. Hence the desired conclusion follow from Theorem~\ref{vasco}.
\end{proof}

By the result of Gruson-Lazardsfeld-Peskine \cite{GLP}  we have $r(I_A)\leq a-2$ if  $\{0,1,\ldots,a-1,a\}\subset A$. This provides another proof of Corollary~\ref{trickyteacher}.

\medskip

Let  $I\in\Ic_{a,b}^1$ and let $c(I)=\inf\{k\: r(I^k)=1\}$. Then Theorem~\ref{related} implies $c(I)\leq g-2$, if $c(I)<\infty$, where $g=\gcd(a,b)$.

\begin{Proposition}
\label{masoomehproves}
Let $2\leq a\leq b$ be integers with $g=\gcd(a,b)$, and  let $j$ be an integer with $1\leq j\leq g-2$.  Then there exists  an ideal $I\in \Ic_{a,b}^1$  such that $c(I)=j$.
\end{Proposition}

\begin{proof}
Let $j$ be any integer with $1\leq j\leq g-2$, and let $A=[0,g-j-1]\union [g-1,g]$. We claim that $c(I_A)=j$. The claim implies the desired conclusion.

Proof of claim: By (\ref{trivial}) it  suffices to show
\begin{enumerate}
\item[(i)] $[0,jg]=jA$;
\item[(ii)] $(j-1)g-j\notin (j-1)A$.
\end{enumerate}

Proof of (i): It is enough to show that $[0,jg]\subseteq jA $. Let $k\in [0,jg]$, we want to show $k\in jA$. We have
\[
jA=\Union_{i=0}^{j}(i[0,g-j-1]+(j-i)[g-1, g])=\Union_{i=0}^{j}([0,i(g-j-1)]+[(j-i)(g-1),(j-i)g]).
\]
Therefore,
\[
jA=\Union_{i=0}^{j}[(j-i)(g-1),i(g-j-1)+(j-i)g].
\]
Let $I_{i}=[(j-i)(g-1),i(g-j-1)+(j-i)g]$. It suffices to show $I_{i}\bigcap I_{i+1}\neq \emptyset$ for all $0\leq i\leq j-1$, which is equivalent to say that
\[
(j-i)(g-1)\leq ((i+1)(g-j-1)+(j-(i+1))g)+1 \quad \text{for all}\quad    0\leq i\leq j-1.
\]
In other words, we need  $i\leq i(g-j-1)$  for all  $0\leq i\leq j-1$. Indeed, this is satisfied,  since $j\leq g-2$. So (i) is proved.

\medskip
Proof of (ii): Suppose that
\[
(j-i)g-j\in(j-1)A=\Union_{i=0}^{j-1}[(j-1-i)(g-1),i(g-j-2)+(j-1-i)g].
\]
Then
\[
(j-1)g-j\in [(j-1-i)(g-1),i(g-j-2)+(j-1-i)g]\quad \text{for some}\quad i\in [0,j-1].
\]
Hence simultaneously we have,
\begin{enumerate}
\item[(1)]$(j-1-i)(g-1)\leq(j-1)g-j$.
\item[(2)]$(j-1)g-j\leq i(g-j-2)+(i-1-i)d$.
\end{enumerate}
(1) implies that $i\geq1$ and (2) implies that $j\geq(j+2)i$, and it is impossible. So  we have a contradiction.
\end{proof}

The following gives us the limit value of $r(I^k)$ for $k\gg 0$.

\begin{Proposition}
\label{hoa*}
 Let $I\in \Ic_{a,a}$ with $a\geq2$, and let $I_A\subset I$ be the equigenerated part of $I$. Then for all $k\geq a-2$,

 \[
r(I^k) = \left\{
\begin{array}{ll}
1,  &   \text{if $xy^{a-1},x^{a-1}y\in I$,}\\
2, & \text{otherwise.}\\
\end{array}
\right. \]
\end{Proposition}

\begin{proof} Suppose first that $xy^{a-1},x^{a-1}y\in I$. Then $xy^{a-1},x^{a-1}y\in I_A$. Therefore, $r(I_A^k)=1$ for $k\geq a-2$ by Theorem~\ref{related}. Now, since $\gcd(A)=1$, Theorem~\ref{redone} implies that $I_A^k=(x,y)^{ka}$ for all $k\geq a-2$. Since $I_A^k\subseteq I^k$, we also have $I^k=(x,y)^{ka}$.This gives us $r(I^k)=1$ for all $k\geq a-2$.

Next assume that $xy^{a-1}\notin I$ or $x^{a-1}y\notin I$. Then  $xy^{a-1}\notin I_A$ or $x^{a-1}y\notin I_A$, and so Theorem~\ref{related}, implies that $r(I_A^k)\geq 2$ for all $k\geq a-2$.
Note that for all $k$, $I_A^k$ is the ideal equigenerated part of $I^k$.  Therefore, Proposition~\ref{turtle} implies that $r(I^k)\geq r(I_A^k)\geq 2$ for all $k\geq a-2$. On the other hand, Proposition~\ref{hoa} yields that $r(I^k)\leq 2$, because $r(I)\leq a-1$ by Theorem~\ref{cold}. This yields the desired conclusion.
\end{proof}

\medskip


\end{document}